\documentclass{amsart}
\usepackage{amssymb,amsmath,amsfonts,amsthm}
\usepackage{graphicx}
\graphicspath{ {./images/} }

\usepackage{psfrag}
\usepackage{mathtools}
\usepackage{color}
\usepackage{todonotes}
\usepackage{enumitem}

\theoremstyle{plain}
\newtheorem{main}{Theorem}

\newtheorem{theorem}{Theorem}[section]
\newtheorem{lemma}[theorem]{Lemma}
\newtheorem{proposition}[theorem]{Proposition}

\theoremstyle{remark}
\newtheorem{remark}[theorem]{Remark}
\newtheorem{definition}{Definition}

\newcommand{\Leb}{\operatorname{vol}}

\newcommand{\C}{\operatorname{C}}

\newcommand{\Gibbs}{\operatorname{Gibbs}}

\newcommand{\diam}{\operatorname{diam}}

           \def\ea{\end{array}}
          \def\ec{\end{center}}
     \def\ed{\end{description}}
        \def\ee{\end{equation}}
       \def\eea{\end{eqnarray}}
     \def\eeaa{\end{eqnarray*}}
 \def\et{\end{thebibliography}}

\def\Phy{\operatorname{PhL}}

\def\Diff{{\rm Diff}}

\def\cG{{\mathcal G}}
\def\cA{{\mathcal A}}

\def\cC{{\mathcal C}}

\def\cU{{\mathcal U}}

\def\cB{{\mathcal B}}

\def\cF{{\mathcal F}}

\def\cP{{\mathcal P}}

\def\cS{{\mathcal S}}

\def\vep{\varepsilon}

\def\TT{{\mathbb T}}

\def\ZZ{{\mathbb Z}}
\def\NN{{\mathbb N}}

\title{Statistical properties of physical-like measures}
\date{\today}
\author{Shaobo Gan, Fan Yang, Jiagang Yang and Rusong Zheng}
\thanks{S.G. is supported by NSFC 11231001 and NSFC 11771025. J.Y. is partially supported by CNPq, FAPERJ, and PRONEX of Brazil and NSFC 11871487 of China.}

\address{School of Mathematical Sciences, Peking University, Beijing 100871, China}
\email{gansb\@@pku.edu.cn}

\address{Department of Mathematics, Michigan State University, East Lansing, Michigan, USA.}
\email{fan.yang-2\@@ou.edu, yangfa31@msu.edu}

\address{Departamento de Geometria, Instituto de Matem\'atica e Estat\'\i stica, Universidade Federal Fluminense, Niter\'oi, Brazil}
\email{yangjg\@@impa.br}

\address{Southern University of Science and Technology, Shenzhen 518055, China}
\email{zhengrs\@@sustech.edu.cn}

\begin{document}

\begin{abstract}
In this paper we consider the semi-continuity of the physical-like measures for diffeomorphisms with dominated splittings. We prove that any weak-* limit of physical-like measures along a sequence of $C^1$ diffeomorphisms $\{f_n\}$ must be a Gibbs $F$-state for the limiting map $f$. As a consequence, we establish the statistical stability for the $C^1$ perturbation of the time-one map of three-dimensional Lorenz attractors, and the continuity of the physical measure for the diffeomorphisms constructed by Bonatti and Viana.

\end{abstract}

\maketitle


\section{Introduction}
Let $f:M\to M$ be a diffeomorphism on some compact Riemannian manifold $M$.
An $f$-invariant probability measure  $\mu$  is a \emph{physical measure} if the set of points $x\in M$ for which the empirical measures $\delta_x^{f,n}$ satisfy
\begin{equation}\label{eq.basin}
\delta_x^{f, n}:=\frac{1}{n}\sum_{i=0}^{n-1}\delta_{f^i(x)}\to \mu \text{ (in the weak-*
sense)}
\end{equation}
has positive volume. This set is called the {\em basin of $\mu$} and is denoted by $B(\mu)$.  We say that an invariant probability $\mu$ is \emph{physical-like} if:
for any small neighborhood $U$ of $\mu$ inside the space of probabilities $\cP(M)$ (not necessarily invariant under $f$) with respect to the
weak-* topology, the set
$$\{x\in M: \text{ there are infinitely many $n_{x,k}\in\NN$ such that } \delta^{f,n_{x,k}}_x\in U\}$$
has positive volume. The set of physical-like measures of $f$ is denoted by $\Phy(f)$. Even though $f$ may not have any physical measure, $\Phy(f)$ is always nonempty, and any physical measure is physical-like (see \cite{CE}).

In this paper, we investigate the properties of physical-like measures under  the
setting of diffeomorphisms with dominated splitting. More precisely, we assume that there exists a
splitting $TM = E \oplus F$ of the tangent bundle that is invariant under the tangent map $Df$, and
satisfies
\begin{equation}\label{eq.dominated}
\|(Df\mid_{F(x)})^{-1})\| \|Df \mid_{E(x)}\| < 1  \text{ at every } x \in M.
\end{equation}
In other words, the bundle $F$ \emph{dominates} the bundle $E$. 

A program
for investigating the physical measures of partially hyperbolic diffeomorphisms or diffeomorphisms with dominated splitting
was initiated by Alves, Bonatti, Viana in \cite{BoV00,ABV00}. Further results can be found in~\cite{Dol00, DVY16,An10,AnV18, AnV,Yang-partial,Yang-expanding,HYY,CYZ} and the references therein.
All these works rely on the Pesin theory, in particular,  the absolute continuity of the Pesin stable lamination. As a result, $f$ has been assumed to be at least $C^{1+\alpha}$.

Below we will introduce a different method to describe the physical and physical-like measures, which focuses more on the Ruelle's inequality and Pesin's entropy formula, and
works even for $C^1$ systems.

\begin{definition}
We say that  a probability measure $\mu$ is a \emph{Gibbs $F$-state} of $f$ if
\begin{equation}\label{e.GibbsF}
h_\mu(f)\geq \int \log \mid \det(Df\mid_{F(x)})\mid d\mu(x),
\end{equation}
where $h_\mu(f)$ is the measure-theoretic entropy of $\mu$.
We denote the space of Gibbs $F$-states by $\Gibbs^F(f)$.
\end{definition} When all the Lyapunov exponents of $\mu$ along
the $F$ bundle  at almost every point are positive, and when the other exponents are non-positive, then combined with Ruelles's inequality \cite{Rue78},
it is easy to see that the previous inequality is indeed an equality, which is known as Pesin's
entropy formula.

The relation between physical-like measures and Gibbs $F$-states is established by Catsigeras, Cerminara, and Enrich in \cite{CCE}:\footnote{In~\cite{CYZ} and~\cite{HYY} a similar result is obtained for the {\em Gibbs $u$-states} of $C^1$ partially hyperbolic diffeomorphisms. Here a measure is called a Gibbs $u$-state  if it satisfies the relation $h_\mu(f,\cF^u)\geq \int \log \mid \det(Df\mid_{E^u(x)})\mid d\mu(x)$, where $h_\mu(f,\cF^u)$ is the {\em partial entropy} along the unstable foliation. Also note that Proposition~\ref{p.GibbsF} does not require $f$ to be partially hyperbolic.}

\begin{proposition}\label{p.GibbsF}
Let $f$ be a $C^1$ diffeomorphism which admits a dominated splitting $E\oplus F$, then
there is a full volume subset $\Gamma$ such that, for any $x\in \Gamma$, any  limit point  $\mu$ of the
sequence $\{\delta^{f,n}_x\}$ belongs to $\Gibbs^F(f)$. Moreover, we have $\Phy(f)\subset \Gibbs^F(f)$.
\end{proposition}

In particular, their result shows that $\Gibbs^F(f)$ is always non-empty; furthermore,  if $f$ admits a unique Gibbs $F$-state $\mu$, then $\mu$ is a physical measure whose basin has full volume.

The relation between physical-like measures and the inequality~\eqref{e.GibbsF} was first discovered by Keller \cite[Theorem 6.1.8]{K} for one-dimensional $C^1$ map with Markov partition. Campbell and Quas
used Keller's result in \cite{CQ} to show that every $C^1$ generic circle expanding map admits a unique physical measure,
whose basin has full volume. The key point of Campbell and Quas' proof is to show that generic expanding map
admits a unique Gibbs $F$-state. Later, Qiu (\cite{Q}) built the same result for uniformly hyperbolic
attractors, and  proved that  $C^1$ generic hyperbolic attractor admits a unique physical measure.

Despite that all the works mentioned above are under $C^1$ generic context, we would like to point out that Proposition~\ref{p.GibbsF} works beyond $C^1$ generic setting. It can be applied to a given diffeomorphism with higher regularity ($C^{1+\alpha}$) where it is used in combination with Pesin's theory. Examples of such applications include  partially hyperbolic diffeomorphisms
with mostly contracting center or with mostly expanding center,  see for instance \cite{Yang-partial,HYY,Yang-expanding,HUY}. 

The table below  summarizes the various properties of $\Phy(f)$ and $\Gibbs^F(f)$. For the precise statements and proofs, see Proposition~\ref{p.1} in the next section.
\begin{center}
\renewcommand{\arraystretch}{1.4} 
    \begin{tabular}{|c|c|c|}
    \hline
    {\hspace{2mm}}  {\hspace{2mm}} & {\hspace{1cm}} $\Phy(f)$ {\hspace{1cm}} &   $\Gibbs^F(f) $  \\\hline
    Existence & True & True \\\hline
    Convexity & False & True    \\\hline  
    Compactness & True  & {\hspace{.5mm}} If $h_{ \boldsymbol{.}}(f)$ is upper semi-continuous {\hspace{.5mm}} \\\hline
    Semi-continuity & \color{blue} Theorem~\ref{m.convergency} below & If $h_\mu( \boldsymbol{\cdot})$ is upper semi-continuous \\\hline
    \end{tabular}
\end{center}

Here note that the compactness and semi-continuity of $\Gibbs^F(f)$ largely depend on the continuity of the metric entropy as a function of the invariant measure and of the diffeomorphism. Examples of such continuity include $C^\infty$ maps (by Buzzi~\cite{B97} and Yomdin~\cite{Yom}), diffeomorphisms away from tangencies (by Liao, Viana and Yang \cite{LVY}), time-one map of Lorenz-like flows (by Pacifico, F. Yang and J. Yang \cite{PYY}) and many more.   However, there are also many counterexamples where the metric entropy function fails to be upper semi-continuous (see~\cite{BCT}, \cite{M} and \cite{N1}).



In this paper, we are going to reveal further connections between the two spaces of measures,
under the context of $C^1$ perturbation theory. More precisely, we will establish a more general continuity for  physical-like measures,
without any extra hypothesis such as $h$-expansiveness.

\begin{main}\label{m.convergency}
Suppose that $f$ is a $C^1$ diffeomorphism which admits a dominated splitting $E\oplus F$, and $\{f_n\}$ is a sequence of diffeomorphisms
converging to $f$ in the $C^1$ topology. Then any weak-* limit of any sequence of physical-like measures $\mu_n$ of $f_n$
is a Gibbs $F$-state of $f$.
\end{main}

As a result, if a diffeomorphism $f$ admits a dominated splitting and has a unique Gibbs $F$-state $\mu$,
then for any $C^1$ nearby diffeomorphism $g$, and for any point $x$ in a full volume subset, any weak-* limit of the empirical measures
$\{\delta_x^{g,n}\}_n$ must be close to $\mu$. In particular, any physical measure of
$g$ (when exists) must be close to $\mu$. This means that the existence of a unique Gibbs $F$-state is an intrinsic
statistically stable property.

As an application of Theorem~\ref{m.convergency}, we will show in Section~\ref{s.examples} that the time-one map  of three-dimensional singular hyperbolic attractors and the example of Bonatti and Viana on $\TT^4$ are statistically stable. More precisely, we will show that for the diffeomorphism of Bonatti and Viana, the unique physical measure  varies continuously in $C^1$ and weak-* topology. We would like to remark that this result could also be obtained from the general criterion of~\cite[Theorem E]{SYY} combined with the careful study of the entropy structure for the Bonatti-Viana maps  in~\cite{CDT}, which already shows that the unique physical measure is ``almost expansive'' (for the precise definition, see~\cite[Definition 2.3]{CDT}).  However, the method in this paper does not rely on this fact.

\begin{remark}
Even though we assume that the dominated splitting is defined on the entire manifold $M$, it is straight forward to check that the Theorem~\ref{m.convergency} remains true when the dominated splitting is defined on a compact invariant set  $\Lambda$. 
In this case the dominated splitting on $\Lambda$, together with the invariant cones, can be extended to a small neighborhood of $\Lambda$, see~\cite[Appendix B]{BDV} for more detail. Then Theorem~\ref{m.convergency} can be applied to a sequence of physical-like measures supported in this neighborhood.
\end{remark}

\section{Properties of physical-like measures and $\Gibbs F$-states}

In this section we collect some properties of $\Phy(f)$ and $\Gibbs^F(f)$.

\begin{proposition}\label{p.1}
Let $f$ be a $C^1$ diffeomorphism with dominated splitting $E\oplus F$, then
\begin{enumerate}
\item $\Phy(f)\subset \Gibbs^F(f)$.
\item $\Phy(f)$ is non-empty and compact.
\item $\Gibbs^F(f)$ is non-empty and convex.
\item $\Gibbs^F(f)$ is compact if $h_\mu(f)$ is upper semi-continuous w.r.t.\,$\mu$.
\item $\Gibbs^F(f)$ varies upper semi-continuously w.r.t.\,$f$ if $h_\mu(f)$ varies upper semi-continuously w.r.t.\,both $f$ and $\mu$. To be more precise, if $f_n\to f$ (in, say, $C^1$ topology) and  $\mu_n\in \Gibbs^{F_n}(f_n)$ with $\mu_n\xrightarrow{\mbox{weak-*}}\mu$ and satisfy
$$\limsup_nh_{\mu_n}(f_n)\le h_\mu(f),$$
then $\mu\in\Gibbs^F(f)$.
\end{enumerate}
\end{proposition}

\begin{proof}
(1) follows from Proposition~\ref{p.GibbsF}.

\noindent For (2), the non-emptiness and compactness follows immediately from the definition of $\Phy(f)$ and the fact that $\cP(M)$ is compact. See also~\cite{CE} for more detail.

\noindent (3) $\Gibbs^F(f)$ is non-empty because of (1) and (2). It is convex since $h_\mu(f)$ is affine in $\mu$.

\noindent To prove (4), take $\mu_n\in\Gibbs^F(f)$ and assume that $\mu_n\to \mu$ in weak-* topology. If $h_\mu(f)$, as a function of $\mu$, is upper semi-continuous, we obtain
\begin{align*}
h_\mu(f)\ge \limsup_n h_{\mu_n}(f) \ge& \limsup_n  \int \log \mid \det(Df\mid_{F(x)})\mid d\mu_n(x)\\
 =& \int \log \mid \det(Df\mid_{F(x)})\mid d\mu(x).
\end{align*}
So $\mu\in \Gibbs^F(f)$.

\noindent For (5) the proof is similar to (4) and omitted. Note that dominated splitting is persistent under $C^1$ topology: if $f_n\xrightarrow{C^1} f$ then $f_n$ has dominated splitting $E_n\oplus F_n$ such that $E_n\to E$ and $F_n\to F$ in the Grassmannian.
\end{proof}

It is also worthwhile to note that $\Phy(f)$ may not be convex, and there exist examples for which $\Phy(f)\subsetneqq \Gibbs^F(f)$. 
To see such an example, consider a uniformly hyperbolic diffeomorphism $f$ with two disjoint transitive attractors $\Lambda_1$ and $\Lambda_2$, each of which supports an ergodic physical measure $\mu_i$, $i=1,2$. Then $\Phy(f) = \{\mu_1, \mu_2\}$ is not convex. Meanwhile, $\Gibbs^F(f) = \{a\mu_1+(1-a)\mu_2: a\in [0,1]\}$, so $\Phy(f)\subsetneqq \Gibbs^F(f)$. See also~\cite{CE} and the discussion following~\cite[Corollary 2]{CCE}.

Finally we would like to point out that neither $\Phy(f)$ nor $\Gibbs^F(f)$ behave well under ergodic decomposition. There are examples (see~\cite{CE}) such that $\Phy(f)$ consists of a single measure which is not ergodic. Meanwhile it is easy to construct examples such that typical ergodic components of a measure $\mu\in\Gibbs^F(f)$ are no longer in $\Gibbs^F(f)$.

\section{Proof of the main theorem}
In this section we prove Theorem~\ref{m.convergency}.
From now on, $\{f_n\}$ is a sequence of $C^1$ diffeomorphisms with $f_n\xrightarrow{C^1} f$. For convenience we will write $f_0 = f$. Denote by $E_n\oplus F_n$ the dominated splitting for $f_n$. We will take $\mu_n\in \Phy(f_n)$ with $\mu_n\xrightarrow{\mbox{weak-*}}\mu$. Then $\mu$ is an invariant probability of $f$.

Let us briefly explain the structure of the proof. Proving by contradiction, we will assume that the limiting measure $\mu$ is not a Gibbs $F$-state. As a result, the metric pressure of $\mu$:
$$
P_{\mu}(f):= h_\mu(f) - \int \log \mid \det(Df\mid_{F(x)})\mid d\mu(x)
$$
must be negative. Then, for a proper finite partition $\cA$ the metric pressure of $f_n$ with respect to the $n$th join of $\cA$ (and note that such join depends on the map $f_n$) is negative, uniformly in $n$: there exists $b>0$, $N>0$ such that for all $n$ large enough:
$$
\frac{1}{N}H_{\mu_n}\left(\bigvee_{i=0}^{N-1}f_n^{-i}(\cA)\right)-\int  \log \mid \det(Df_n\mid_{F_n(x)})\mid d\mu_n <-b < 0.
$$
This step is carried out in Section~\ref{s.3.1} and~\ref{s.3.2}.

From here the proof largely follows the idea of the variational principle~\cite{Wal}. We will consider the following good set: 
$$
\cG_m^n = \{x: \delta_x^{f_n, m}\in U_n\}
$$
where $U_n$ is a small neighborhood of $\mu_n$ in the space of probability measures. Using the pressure estimate above, we will show in Section~\ref{s.3.3} that the volume of $\cG_m^n$, when restricted to any disk tangent to local $F_n$-cone with dimension equal to $\dim F$, is of order $e^{-bm}$; furthermore, this estimate can be made uniform in $n$.
Then it follows that for Lebesgue almost every point, the empirical measures $\delta_x^{f_n,m}$ can only be in $U_n$ for finitely many $m$'s, contradicting with the choice of $\mu_n\in\Phy(f_n)$.

To this end, we will assume from now on that $\mu\notin \Gibbs^F(f)$. To simplify notation we write
$$
\phi_f^F(x) = -\log \mid \det (Df\mid_{F(x)}) |,
$$
then there exists $a>0$ such that
\begin{equation}\label{e.gap}
h_\mu(f) + \int \phi_f^F(x)\, d\mu(x)\le -a <0.
\end{equation}
Also note that the definition of the function $\phi_f^F(x)$ can be extended to any subspace  $\tilde F(x)\subset T_xM$.

\subsection{A $C^1$ neighborhood of $f$}\label{s.3.1}

We denote by $d^G$ the distance in the Grassmannian manifold. By the continuity of $\phi_f^F$ and the compactness of the Grassmannian, we can choose a $C^1$ neighborhood $\cU$ of $f$ and $\delta_0>0$ small enough, with the following property:

\begin{quotation}
For any $g\in \cU$, $x,y\in M$ with $d(x,y)<\delta_0$, and any subspaces $\tilde F(x)\subset  T_xM$, $\tilde F(y)\subset T_yM$ with $\dim \tilde F(x) = \dim \tilde F(y) = \dim F_0$ and
$$
d^G(F_g(x), \tilde F(x))<\delta_0,\,\,\, d^G(F_g(y), \tilde F(y))<\delta_0
$$
where $E_g\oplus F_g$ is the dominated splitting of  $g$ which is the continuation of $E\oplus F$, one has
\begin{equation}\label{e.delta0}
|\phi_g^{\tilde{F}}(x) - \phi_g^{\tilde{F}}(y)|<\frac{a}{1000}.
\end{equation}

\end{quotation}

We further assume that $\delta_0$ is small enough, such that for every point $x\in M$ the exponential map $\exp_x: T_xM \to M$ sends the $\delta_0$-ball $B_{\delta_0}(0_x)\subset T_xM$ diffeomorphically onto its image.

For any $\delta>0$, we denote by $\C_{\delta}(F_g(x))\subset T_xM$ the $\delta$-cone around $F_g(x)$:
$$
\C_{\delta}(F_g(x))=\{v\in T_xM: |v_{E_g}|\leq \delta |v_{F_g}| \,\mbox{ for } v = v_{E_g}+ v_{F_g}\in E_g\oplus  F_g \}.
$$
By the dominated assumption, the cone field $\C_{\delta}(F_g)$ is invariant under the iteration of $Dg$, i.e., there is $0<\lambda<1$ independent of $\delta$ such that 
$$Dg(\C_{\delta}(F_g(x)))\subset \C_{\lambda \delta}(F_g(g(x))).$$
For $\delta_0$ satisfying~\eqref{e.delta0} above, we will refer to $$
\cC_{\delta_0}(F_g(x))=\exp_x\big(\C_{\delta_0}(F_g(x))\cap B_{\delta_0}(0_x)\big)
$$ 
as the local $F_g$ cone on the underlying manifold $M$. Note that the local $F_g$ cones are invariant in the following sense: there exists  $\delta_1\in (0,\delta_0)$ small enough such that for all $x\in M$ one has 
$$
g\left(\cC_{\delta_0}(F_g(x))\cap B_{\frac{\delta_1}{\|g\|_{C^1}}}(x)\right)\subset \cC_{\delta_0}\big(F_g(g(x))\big)\cap B_{\delta_1}(g(x)).
$$

From now on, $\delta_0$ and $\delta_1$ will be fixed.

\begin{definition}
Given $g\in\cU$, an embedded submanifold $K\subset M$ is said to be {\em tangent to local $F_g$ cone,} 
if for any $x\in K$ one has
$$
K \subset  \cC_{\delta_0}(F_g(x))\cap  B_{\delta_1}(x).\footnote{We slightly abuse notation and use $B_{\delta_0}(\cdot)$ both for balls in $T_xM$ and in $M$; one could easily tell the difference by looking at the center. }
$$
\end{definition}

The following simple lemma is taken from \cite[Lemma 2.3]{PYY}.

\begin{lemma}\label{l.localvolume}
There is a constant $L > 0$ such that for any $g\in \cU$ and for every $x\in M$ and any disk
$D$ tangent to local $F_g$ cone, we have $vol(D) < L$.
\end{lemma}

Writing
$$
B_{\delta,n}(x,g):= \{y\in M: d(g^i(x), g^i (y))< \delta, i = 0,\ldots, n-1\}
$$
for the $(\delta, n)$-Bowen ball around $x$. By the contraction of the cone filed on the tangent space, it follows that

\begin{lemma}\label{l.volumebound}
If $K\subset M$ is a disk tangent to local $F_g$ cone with dimension $\dim(F_g)$, then for any $x\in K$ and $n\ge 0$, 
$$g^n\left(K\cap B_{\frac{\delta_1}{\|g\|_{C^1}},n}(x,g)\right)$$
is still tangent to local $F_g$ cone. Moreover,
\begin{equation}\label{e.volumebound}
\Leb_{g^n(K)}\left(g^n\left(K\cap B_{\frac{\delta_1}{\|g\|_{C^1}},n}(x,g)\right)\right)\le L.
\end{equation}
\end{lemma}

\begin{proof}The first part of the lemma follows from the forward invariance of the  local cone and induction. 
The second part is a consequence of Lemma~\ref{l.localvolume}. 
\end{proof}

From now on, we take $\delta_2= \frac{\delta_1}{\sup_{g\in \cU}\{\|g\|_{C^1}\}}$ which will be the size of the Bowen balls and separated sets.

\subsection{A finite partition}\label{s.3.2} The goal of this section is to rewrite~\eqref{e.gap} in terms of the information entropy $H_{\nu}$ of the finite join (under the iteration of the perturbed maps $f_n$) of a  finite partition; here $\nu$ is a probability measure (not necessarily invariant under $f$ or $f_n$) that is close to some $\mu_n$.
To this end, let $\delta_0$ and $\cU$ be given in the previous section, and recall that  $\mu_n\to \mu$ in weak-* topology. We also write $\mu_0 = \mu$.
Fix $\cA$ a finite, measurable partition of $M$ with $\diam(\cA)<\delta_2$, such that
$\mu_i(\partial \cA)=0$ for all $i=0,1,\ldots$. The existence of such a partition follows from the
fact that there are at most countable disjoint sets with positive $\mu_i$ measure for each $i$, thus for any
point $x$, there is a ball with radius $r_x$ arbitrarily small such that the boundary of this ball
has vanishing $\mu_i$ measure for any $i$. Each ball and its complement form a partition, we
can take $\cA$ as the refinement of finitely many such partitions.

Moreover, we can take $\cA$
to be fine enough, such that:
\begin{equation}\label{eq.smalldiameter}
h_\mu(f,\cA)+\int \phi^F_f d\mu <-\frac{999}{1000}a.
\end{equation}
Then there is $N$ large enough, such that
\begin{equation}\label{eq.largestep1}
\frac{1}{N}H_\mu\left(\bigvee_{i=0}^{N-1}f^{-i}(\cA)\right)+\int \phi^F_f d\mu <-\frac{998}{1000}a.
\end{equation}

We would like to replace $\frac{1}{N}H_\mu\left(\bigvee_{i=0}^{N-1}f^{-i}(\cA)\right)$ by $\frac{1}{N}H_{\mu_n}\left(\bigvee_{i=0}^{N-1}f_n^{-i}(\cA)\right)$. This creates an extra difficulty since the partition in question depends on $f_n$. To solve this issue, we introduce the following lemma, whose proof is standard in the
measure theory and is thus omitted.

\begin{lemma}\label{l.measure}
Let $\mu,\mu_i$, $i=1,2\ldots$ be probability measures such that $\mu_n\xrightarrow[n\to\infty]{\mbox{weak-*}}\mu$. Let $A,A_i,i=1,2,\ldots$ be a sequence of measurable sets with the following properties:
\begin{enumerate}
\item $\tilde{A}:=\mbox{int}(\overline{A})$, $\tilde{A}_n:=\mbox{int}(\overline{A}_n)$ satisfy that for every compact set $K\subset \tilde{A}$, there exists $N_K>0$ such that $K\subset \tilde{A}_n$ for all $n>N_K$;
\item the above property holds with $A$ and $A_n$ replaced by $A^c$ and $A_n^c$;
\item $\mu(\partial A) = \mu_i(\partial A_i)=0$, $i=1,2,\ldots$.
\end{enumerate}
then we have $\lim_{n\to\infty}\mu_n(A_n)= \mu(A)$.
\end{lemma}

As an immediate application, we have:

\begin{lemma}\label{l.partition}
For $N>0$ fixed, we have
$$
\lim_{n\to\infty}\frac{1}{N}H_{\mu_n}\left(\bigvee_{i=0}^{N-1}f_n^{-i}(\cA)\right) = \frac{1}{N}H_{\mu}\left(\bigvee_{i=0}^{N-1}f^{-i}(\cA)\right)
$$
\end{lemma}

\begin{proof}
Elements of $\bigvee_{i=0}^{N-1}f^{-i}(\cA)$ have the form:
$$
A= \bigcap_{i=0}^{N-1} f^{-i}(B_i)
$$
for some sequence $\{B_i\in \cA\}_{i=0}^{N-1}$. Given such a sequence, we denote by
$$
A_n = \bigcap_{i=0}^{N-1} f_n^{-i}(B_i).
$$
Since $f_n$ converges to $f$ in $C^1$ topology and elements of $\cA$ are finite intersections of open balls $B_{r_{x_k}}(x_k)$ and their complements, (1) and (2) of Lemma~\ref{l.measure} are satisfied by the sets $A$ and $A_n$. For (3) of Lemma~\ref{l.measure}, observe that $\mu_n(\partial \cA) = 0$ implies that $\mu_n\left(\bigcup_{i=0}^{N-1}f_n^{-i}(\partial \cA)\right) = 0$ for all $n$, and the same holds for $\mu$ and $f$. It follows that $\mu(\partial A) = \mu_n(\partial A_n) = 0$.

Now we can apply the previous lemma to get $\lim_n \mu_n(A_n) = \mu(A)$. In particular,
$$
-\mu_n(A_n) \log \mu_n(A_n) \to -\mu(A) \log \mu(A).
$$
Summing over all elements of $\bigvee_{i=0}^{N-1}f^{-i}(\cA)$  (and keep in mind that this is a finite partition for fixed $N$) and divide by $N$, we obtain the desired result.
\end{proof}


Combine Lemma~\ref{l.partition} with~\eqref{eq.largestep1} and use the continuity of $\phi_{f_n}^{F_n}$, we conclude that there exists $N_1\in\NN$ such that
\begin{equation}\label{eq.largestep2}
\frac{1}{N}H_{\mu_n}\left(\bigvee_{i=0}^{N-1}f_n^{-i}(\cA)\right)+\int \phi^{F_n}_{f_n} d\mu_n <-\frac{996}{1000}a,
\end{equation}
for all $n>N_1$.

Note that for each $n$ and $N$, the set $\bigcup_{i=0}^{N-1}f_n^{-i}(\partial\cA)
$ is closed, so its measure varies upper semi-continuously with respect to probability measures (not necessarily invariant by any of $f_n$).
Fix $\varepsilon>0$ small enough. For each $n>N_1$, we can take a small convex neighborhood $U_n \subset \cP(M)$ of $\mu_n$, such that
for any $\nu\in U_n$,
\begin{equation}\label{eq.nu}
\begin{split}
\nu&\left(\bigcup_{i=0}^{N-1}f_n^{-i}(\partial\cA)\right) <\vep,\,\mbox{ and consequently}\\
\frac{1}{N}H_\nu&\left(\bigvee_{i=0}^{N-1}f_n^{-i}(\cA)\right) \le \frac{1}{N}H_{\mu_n}\left(\bigvee_{i=0}^{N-1}f_n^{-i}(\cA)\right) + \frac{1}{1000}a.\\
\\
\end{split}
\end{equation}

By the continuity of $\phi_{f_n}^{F_n}$, we can shrink $\vep$ and finally obtain
\begin{equation}\label{eq.nu2}
\frac{1}{N}H_\nu\left(\bigvee_{i=0}^{N-1}f_n^{-i}(\cA)\right)+\int \phi^{F_n}_{f_n} d\nu <-\frac{994}{1000}a,\,\,\mbox{ for all } \nu \in U_n.\\
\end{equation}

\subsection{From pressure to the measure of the good set}\label{s.3.3}
In this subsection we assume that $I$ is a smooth disk with dimension $\dim F$ that is tangent to local $F_{n}$ cone (to simplify notation we write $F_n = F_{f_n}$), for some $n=0,1,2,\ldots$. Recall that  every $\mu_n$ is a physical-like measure of $f_n$. As a result, there is a positive volume subset $\Lambda_n$ such that for every
$x\in \Lambda_n$, there is a sequence $\{i_k\}$ such that the empirical measures satisfy
\begin{equation}\label{eq.U_n}
\delta^{f_n,i_k}_x\in U_n
\end{equation}
for any $k$, where $U_n$ is the neighborhood of $\mu_n$ in $\cP(M)$ that we chose in the previous subsection such that
\eqref{eq.nu} and \eqref{eq.nu2} hold.

To obtain a contradiction, for $m\in\NN$, we define
$$
\cG^{I,n}_m=\{x\in I\cap \Lambda_n: \delta^{f_n,m}_x\in U_n\}.
$$ The goal is to show that $\cG^{I,n}_m$ have small measure with respect to the Riemannian volume on $I$, uniformly in $n$. 

For this purpose, we let $E^n_m$ be a maximal $(\delta_2,m)$-separated set of $\cG^{I,n}_m$ w.r.t. the map $f_n$. Here we drop the dependence of $E^n_m$ on $\delta_2$ since it is fixed throughout the paper.
Consider the following probability measure  $\sigma^n_m$ supported on $E^n_m$:
$$\sigma^n_m=\frac{\sum_{z\in E^n_m} e^{S^{f_n}_m \phi^{F_n}_{f_n}(z)}\delta_z}{\sum_{z\in E^n_m} e^{S^{f_n}_m \phi^{F_n}_{f_n}(z)}}.$$
Here $S^{f_n}_m\phi$ is the Birkhoff sum of $\phi$ w.r.t. the map $f_n$, i.e., $S^{f_n}_m\phi = \sum_{i=0}^{m-1}\phi\circ f_n^i$.

By the convexity of the neighborhoods $U_n$ of $\mu_n$ and the definition of $\cG^{I,n}_m$,  the measure
$$\mu^n_m=\frac{1}{m}\sum_{i=0}^{m-1} (f_n^i)_* \sigma^n_m$$
is a convex combination of $\delta_x^{f_n,m}$ for $x\in \cG^{I,n}_m$ and thus is contained in $U_n$.
Denote by
$$P^n_m=\frac{1}{m}\log \sum_{z\in E^n_m} e^{S^{f_n}_m\phi^{F_n}_{f_n}(z)}$$
the pressure of the separated set $E^n_m$ and $\overline{P}^n=\limsup _m P^n_m$ the limiting pressure for each $f_n$.

\begin{lemma}\label{l.boundedppresure} For $n>N_1$, we have
$$\overline{P}^n<-\frac{994}{1000}a.$$
\end{lemma}
\begin{proof}
The proof is motivated by the proof of variational principle~\cite{Wal}.

For any $l\geq 2N$, write $a(j) = [\frac{l-j}{N}]$ for $0\le j \le N-1$. Here $N>0$ is the integer chosen in the previous section such that~\eqref{eq.largestep1} to~\eqref{eq.nu2} holds. Then for each $j$ we have
$$\bigvee_{i=0}^{l-1} f_n^{-i}\cA=\bigvee_{r=0}^{a(j)-1} f_n^{-(rN+j)}\left(\bigvee_{i=0}^{N-1}f_n^{-i}\cA\right)\vee \bigvee_{i\in \cS} f_n^{-i}\cA,$$
where $\cS$ is a subset of $\{0,1,\ldots,l-1\}$ with $\# \cS\leq 2N$. Therefore
\begin{equation*}
\begin{split}
&mP^n_m= \log \sum_{z\in E^n_m} e^{S^{f_n}_m\phi^{F_n}_{f_n}(z)}=H_{\sigma^n_m}\left(\bigvee_{j=0}^{m-1}f_n^{-j}\cA\right)+\int S^{f_n}_m \phi^{F_n}_{f_n}d\sigma^n_m\\
&\leq \sum_{r=0}^{a(j)-1} H_{\sigma_m^n}\left(f_n^{-(rN+j)} \bigvee_{i=0}^{N-1}f_n^{-i}\cA\right)+H_{\sigma^n_m}\left(\bigvee_{i\in \cS}f_n^{-i}\cA\right)+\int S^{f_n}_m \phi^{F_n}_{f_n}d\sigma^n_m\\
&\leq \sum_{r=0}^{a(j)-1} H_{(f_n^{rN+j})_*\sigma^n_m}\left(\bigvee_{i=0}^{N-1}f_n^{-i}\cA\right)+2N\log \#\cA + \int S^{f_n}_m \phi^{F_n}_{f_n}d\sigma^n_m.\\
\end{split}
\end{equation*}

Summing over $j$ from $0$ to $N-1$:
$$NmP^n_m\leq \sum_{r=0}^{m-1} H_{(f_n^r)_*\sigma^n_m}\left(\bigvee_{i=0}^{N-1}f_n^{-i}\cA\right)+2N^2\log \#\cA  + N\int S^{f_n}_m \phi^{F_n}_{f_n}d\sigma^n_m.$$
Dividing by $mN$ yields
\begin{align*}
P^n_m\leq& \frac{1}{N}H_{\mu^n_m}\left(\bigvee_{i=0}^{N-1}f_n^{-i}\cA\right)+\frac{2N}{m}\log \#\cA  + \int  \phi^{F_n}_{f_n}d\mu^n_m\\
\le& -\frac{994}{1000}a + \frac{2N}{m}\log \#\cA ,
\end{align*}
where the second inequality follows from~\eqref{eq.nu2}.

Sending $m$ to infinity, we conclude the proof.
\end{proof}

The main result of this section is the following lemma:

\begin{lemma}\label{l.volume} For every $n>N_1$ and every $(\dim F)$-dimensional disk $I$ that is tangent to local $F_n$ cone, we have
$$\limsup_{m\to\infty} \frac{1}{m} \log \Leb_{I}(\cG^{I,n}_m)<\overline{P}^n+\frac{1}{1000}a<-\frac{993}{1000}a.$$
\end{lemma}

\begin{proof}
By the choice of $\delta_0$, $\delta_2$ and~\eqref{e.delta0}, we obtain
\begin{equation*}
\begin{split}
&\Leb_{I}(\cG^{I,n}_m)\leq \sum_{z\in E^n_m} \Leb_{I}(B_{\delta_2,n}(z,f_n))\\
&\leq \sum_{z\in E^n_m} \Leb_{f^m_n(I)}(f_n^m(B_{\delta_2,n}(z,f_n)))\cdot \mid\det(Df_n^{-m}\mid_{F_n(f_n^m(z))})\mid (e^{\frac{a}{1000}})^m.\\
\end{split}
\end{equation*}

By Lemma~\ref{l.volumebound}, the previous inequality is bounded by
$$\Leb_{I}(\cG^{I,n}_m)\leq L \sum_{z\in E^n_m} \mid\det(Df_n^{-m}\mid_{F_n(f_n^m(z))})\mid (e^{\frac{a}{1000}})^m.$$
Thus
$$\frac{1}{m}\log \Leb_{I}(\cG^{I,n}_m) \leq \frac{1}{m}\log L+ P^n_m +\frac{a}{1000},$$
and
$$\limsup_m \frac{1}{m} \log \Leb_{I} (\cG^{I,n}_m) \leq \limsup_m  P^n_m +\frac{a}{1000}=\overline{P}^n+\frac{a}{1000}<-\frac{993}{1000}a.$$
\end{proof}

\subsection{Proof of Theorem~\ref{m.convergency}}

Fix any $n>N_1$. Recall that $\Lambda_n$ is a positive volume subset such that \eqref{eq.U_n} holds. We take a smooth foliation box $\cB: I^{\dim E}\times I^{\dim F}\to M$ such that $\Leb(\cB\cap \Lambda_n) > 0$ and
for any $a\in I^{\dim E}$, $\cB(a,\cdot)$ maps $\{a\}\times I^{\dim F} $ to a disk $I_a$ that is tangent to  local $F_{n}$ cone. 
Since the foliation chart is smooth, by Fubini theorem, there is $a_n\in I^{\dim E}$ such that the corresponding disk $I_n=I_{a_n}$ satisfies
$\Leb_{I_n}(\Lambda_n)>0$.

On the other hand, Lemma~\ref{l.volume} applied to $I_n$ yields
$$
\limsup_m\frac{1}{m}\log\Leb_{I_n}(\cG_m^{I_n,n}) < -\frac{993}{1000}a < 0.
$$
In particular, this means that
$$
\sum_{m=1}^\infty \Leb_{I_n}(\cG^{I_n,n}_m)<\infty.
$$
By the Borel-Contelli Lemma, we have
$$
\Leb_{I_n}\{x\in I_n: \delta^{f_n,m}_x\in U_n \mbox{ infinitely often}\} =0.
$$
However, this contradicts with the choice of $a_n$ such that $\Leb_{I_n}(\Lambda_n)>0$. We conclude the proof of Theorem~\ref{m.convergency}.

\section{Examples\label{s.examples}}
In this section, we will provide two examples where the metric entropy function is not upper semi-continuous,  yet our main result still applies.
Observe that in these cases, one cannot expect to obtain the continuity of Gibbs $F$-states only from its definition.

\subsection{Statistical stability of singular hyperbolic attractors}

We consider $C^1$ perturbations for the time-one map of singular hyperbolic attractors on three-dimensional manifolds. Let $X$ be a $C^2$ vector field on a compact boundaryless 3-manifold $M$ and $\phi_t$ be the flow induced by $X$.
An attractor $\Lambda$ is called {\em singular hyperbolic}, if all the singularities in $\Lambda$ are hyperbolic, and if there is a dominated splitting for $\phi_t$:
$$
T_\Lambda M = E^s\oplus F^{cu}
$$
with $\dim E^s = 1$, such that $D\phi_t|_{E^s}$ is uniformly contracting, and $D\phi_t|_{F^{cu}}$ is {\em volume expanding}: there exist $C>0$ and $\lambda>1$ such that
$$
|\det D\phi_t|_{F^{cu}}(x)|\ge C\lambda^t
$$
for all $x\in\Lambda$ and $t>0$. Note that this condition prevents trivial measures (i.e., Dirac measure of a singularity or trivial measures on a periodic orbit) to be Gibbs $F^{cu}$-states. Examples of singular hyperbolic attractors include the famous Lorenz attractor. We invite the readers to the book~\cite{ArPa10} for a comprehensive study on this topic.

It is proven in~\cite[Theorem B, C, Corollary 2]{APPV} that every singular hyperbolic attractor has a unique physical measure $\mu$. Moreover, $\mu$ is ergodic, hyperbolic (meaning that $\mu$ has a unique zero Lyapunov exponent which is given by the flow direction), fully supported on $\Lambda$, has absolutely continuous conditional measures on the center-unstable manifolds (these are the images of the Pesin strong unstable manifolds under the flow), and satisfies the entropy formula:
$$
h_\mu(\phi_1) = \int\log |\det D\phi_1|_{F^{cu}}|\,d\mu.
$$
Note that $\phi_1$ is the time-one map of the flow. Denote by $U$ the attracting neighborhood of $\Lambda$. Since $\phi_1$ is $C^2$, by Ledrappier-Young formula~\cite{LY}, we see that $\mu$ is the unique Gibbs $F^{cu}$-state:
$$
\Gibbs^{F^{cu}}(\phi_1|_U) = \{\mu\}.
$$

By Theorem~\ref{m.convergency}, we obtain the statistical stability for $C^1$ perturbations of the time-one map:
\begin{theorem}Assume that $\Lambda$ is a singular hyperbolic attractor for a 3-dimensional flow $\phi_t$ with attracting neighborhood $U$.
Let $\{f_n\}$ be a sequence of $C^1$ diffeomorphisms converging to $\phi_1$ in $C^1$ topology, and $\mu_n$ be a physical-like measure of $f_n$ supported in $U$. Then we have $\mu_n\xrightarrow{weak^*}\mu$, where $\mu$ is the unique physical measure of $\phi_t$ on $\Lambda$.

\end{theorem}
We remark that 
$f_n$ need not be the time-one map of a vector field $X_n$ that is $C^1$ close to $X$. In this case, since the bundle $F^{cu}$
admits no further domination, it is possible to create homoclinic tangency after $C^1$ perturbation, see Gourmelon~\cite[Theorem~3.1]{Gou10},
 Pujals, Sambarino~\cite{PS00} and Wen~\cite{Wen02} for
previous results along this direction. Following the work of Newhouse~\cite{N1} (see also~\cite{BCT} for a refined construction) local horseshoes with large entropy can be created, which prevents the metric entropy from being upper semi-continuous\footnote{Note that the robust $h$-expansiveness (which implies the upper semi-continuity of $h_\mu(X)$) proven in~\cite{PYY} applies only to nearby flows.}.

\subsection{The example of Bonatti and Viana on $\TT^4$}
Following Ma\~n\'e's study~\cite{Ma78} of derived from Anosov diffeomorphisms on $\TT^3$,
Bonatti and Viana constructed in~\cite{BoV00} (see also \cite{Ta}) a family of robustly transitive diffeomorphisms without any hyperbolic direction. Their examples are obtained by  perturbing a linear Anosov diffeomorphism $A: \TT^4\circlearrowleft$ near two fixed points $p$ and $q$.

Let us briefly recall the construction of the example. Let $A\in \operatorname{SL}(4, \ZZ)$  be a linear Anosov diffeomorphism with four distinct real eigenvalues:
$$
0<\lambda_1<\lambda_2<1/3< 3<\lambda_3<\lambda_4,
$$
and let $p$, $q$ be two fixed points of $A$. We fix some $r>0$ small enough and consider the following perturbation of $A$, and note that such perturbations are $C^0$ small but $C^1$ large:
\begin{enumerate}
\item outside $B_r(p)$ and $B_r(q)$ the map is untouched;
\item in $B_r(p)$ the fixed point $p$ undergoes a pitchfork bifurcation in the direction corresponding to $\lambda_2$; this changes the stable index of $p$ from 2 to 1 and creates two new fixed points inside $B_r(p)$ with stable index $2$, which we denote by $p_1$ and $p_2$;
\item a small perturbation near $B_{r'}(p_1)\subset B_r(p)$ makes the contracting eigenvalues complex;
\item repeat steps (2) and (3) in $B_r(q)$ for $A^{-1}$.
\end{enumerate}
Write
$$
\lambda_{cs} = \sup\{\log\|Df_{BV}\mid_{E^{cs}(x)}\|:x\in B_r(p)\}>0,
$$
$$
\lambda_{cu} = \sup\{\log\|Df^{-1}_{BV}\mid_{E^{cu}(x)}\|:x\in B_r(q)\}>0,
$$
$$
\lambda =\max\{\lambda_{cs},\lambda_{cu}\}.
$$
Choosing $r$, $\lambda>0$ small enough (we refer to~\cite{BoV00} for full detail and~\cite{BF} for a refined construction), we obtain a diffeomorphism which we denote by $f_{BV}$, such that:
\begin{itemize}
\item there exists an open neighborhood $\cU_{BV}\subset \Diff^1(\TT^4)$ of $f_{BV}$ such that every $g\in\cU_{BV}$ is transitive;
\item $g\in\cU_{BV}$ admits a dominated splitting $T\TT^4 = E^{cs}\oplus E^{cu}$ with $\dim E^{cs} = \dim E^{cu} =2$; moreover, $E^{cs}$ and $E^{cu}$ cannot be further split into one-dimensional invariant subbundles;
\item $E^{cs}$ and $E^{cu}$ are integrable (this requires the refined construction in~\cite{BF});
\end{itemize}

The following theorem is proven in~\cite{CDT}.
\begin{theorem}\label{t.BV}\cite[Theorem B]{CDT}
For $r,\lambda>0$ small enough, there exists a $C^1$ neighborhood $\cU$ of $f_{BV}$ such that every $g\in \cU\cap \Diff^2(\TT^4)$ has a unique physical measure $\mu_g$ which is the unique equilibrium state for the potential $\varphi_g = -\log \det (Dg\mid_{E^{cu}})$, and satisfies
$$
P(\varphi_g,g) = h_{\mu_g}(g) + \int \varphi_g\,d\mu_g=0.
$$
\end{theorem}
Combining with Theorem~\ref{m.convergency}, we obtain the continuity of the physical measures for the example of Bonatti and Viana:

\begin{theorem}
Let $\cU$ be the neighborhood of $f_{BV}$ given by Theorem~\ref{t.BV}. Then restricted to $\mathcal{U}\cap\mathrm{Diff}^2(\mathbb{T}^4)$, it satisfies that $\mu_g$ varies continuously (in the weak-* topology) with respect to $g$ in the $C^1$ topology.
\end{theorem}
\begin{proof}
The previous theorem states that
$$
\Gibbs^{E^{cu}}(g) = \{\mu_g\}
$$
for $g\in \cU\cap\Diff^2(\TT^4)$.
Let $g_n,g$ be $C^2$ diffeomorphisms in $\cU$, with $g_n\xrightarrow{C^1}g$. Then Theorem~\ref{m.convergency} shows that $\mu_{g_n}\xrightarrow{\mbox{weak-*}}\mu_g$, as desired.
\end{proof}

Similar to the previous example, since the bundles $E^{cs}$ and $E^{cu}$
admit no further domination, it is possible to create homoclinic tangency after $C^1$ perturbation and therefore the metric entropy is 
not upper semi-continuous. See also~\cite{CDT}, particularly Lemma 6.11, for the characterization on the refined entropy structure for $g\in\cU_{BV}$.

\end{document}